\documentclass[10]{article}
\usepackage{geometry}[a4wide]
\usepackage{amssymb} 
\usepackage{latexsym} 
\usepackage{amsmath,empheq}
\usepackage{mathtools} 
\usepackage{thmtools}
\usepackage{tabls} 
\usepackage{graphicx}
\usepackage{todonotes}
\usepackage{overpic}
\usepackage{subcaption}
\usepackage{dutchcal}
\usepackage{bbm}
\usepackage[bottom]{footmisc}
\usepackage[english]{babel}
\usepackage{centernot}
\usepackage{stmaryrd}
\usepackage[utf8]{inputenc}
\usepackage{color}
\usepackage{epsfig}
\usepackage{graphicx}
\usepackage[all]{xy}
\input xy
\usepackage[mathscr]{eucal} %use \mathsrc{X} to type eucal X
\usepackage{enumerate}
\usepackage{enumitem}
\usepackage{accents}
\newtheorem{theorem}{Theorem}

\newtheorem{remark}[theorem]{Remark}

\newtheorem{lemma}{Lemma}

\newenvironment{proof}{\noindent{\bf Proof.}}%
{\hspace*{\fill}$\Box$\par\vspace{4mm}}

\newenvironment{claimproof}[1]{\par\noindent\emph{Proof of}\space\textbf{#1.}}{\hfill $\checkmark$\par\vspace{.5mm}}

\newenvironment{theoremproof}[1]{\par\noindent\textbf{Proof of Theorem}\space\textbf{#1.}\!}{\hfill$\Box$\mybreak}

\def\cadre{$$\vcenter\bgroup\advance\hsize by -2em\noindent
	\refstepcounter{equation}(\theequation)~\ignorespaces}
\makeatletter
\def\endcadre{\egroup\eqno$$\global\@ignoretrue}

\newcommand{\mybreak} {\par\vspace{2mm}\noindent}
\newcommand\mydfrac[2]{{\displaystyle\frac{#1}{#2}}}

\makeatletter
\def\imod#1{\allowbreak\mkern10mu({\operator@font mod}\,\,#1)}
\makeatother

\newcommand{\comment}[1]{}

\newcommand{\E} {\mathbb{E}}

\newcommand{\F} {\mathcal{F}}

\newcommand{\prd} {\mathcal{P}}

\newcommand{\Pm} {\mathbb{P}}
\newcommand{\G} {\mathbb{G}}
\newcommand{\GG} {\mathcal{G}}

\newcommand{\Ea}[1] {\E\left(#1\right)}

\newcommand{\lmax} {\lambda_{\max}}
\newcommand{\lmin} {\lambda_{\min}}
\newcommand{\ap}[2] {#1^{(#2)}}
\newcommand{\ma}[1] {{#1}^{(s)}_n}

\newcommand{\mS}[1] {{#1}^{(S)}}

\newcommand{\mt}[3] {{#1}^{(#2)}_{#3}}
\newcommand{\nn}{\widetilde{n}}

\newcommand{\ber}{{\rm ber}}
\newcommand{\bino}{{\rm bin}}
\newcommand{\bin}[2] {\bino(#1,#2)}

\newcommand{\size}[1] {\mathsf{e}(#1)}

\newcommand{\diagg}{\textrm{diag}}
\newcommand{\diago}[1]{\diagg\left(#1\right)}

\usepackage{multicol}
\usepackage{tikz}
\usepackage{pgfplots}
\usepackage{dsfont}

\usepackage{ulem}
\begin{document}

\title{The Critical Patch Size Problem in Random Graphs} 
\author{
	Nicola Apollonio\thanks{Consiglio Nazionale delle Ricerche, Via dei Taurini 19, 00185 Roma, Italy.
		e-mail: \tt{nicola.apollonio@cnr.it}} 
	\and
	Veronica Tora\thanks{Consiglio Nazionale delle Ricerche, Via dei Taurini 19, 00185 Roma, Italy.
		e-mail: \tt{v.tora@iac.cnr,it, veronica.tora2@unibo.it}} 
	\and
	Davide Vergni\thanks{Consiglio Nazionale delle Ricerche, Via dei Taurini 19, 00185 Roma, Italy.
		e-mail: \tt{davide.vergni@cnr.it}}}

\maketitle

\begin{abstract}
The problem of {\it critical patch size} --- a threshold condition for
population persistence --- is investigated in the context of discrete
habitats, modeled as graphs with a distinguished subset of vertices 
acting as sinks. These sinks impose boundary-like constraints analogous to
Dirichlet conditions in continuous domains. The population proliferates
locally at the vertices and diffuse across the network through the
graph Laplacian. In the sinks the population cannot survive. The
Dirichlet eigenvalue of the habitat is defined as the smallest
eigenvalue of the principal submatrix of the Laplacian obtained by
removing the rows and columns associated with sink vertices. This
spectral parameter governs the habitat's viability: survival occurs
when the Dirichlet eigenvalue of the habitat lies below a critical
reaction-to-diffusion ratio.

We study survival conditions for a sequence of random habitats built on
binomial random graphs. We establish a law of large numbers for the
corresponding sequence of Dirichlet eigenvalues and prove the
emergence of a sharp threshold phenomenon: with high probability, a
large random habitat is either viable or non-viable, depending on
whether the reaction-to-diffusion ratio lies below or above this
threshold. Our results provide the first general spectral theory for
critical patch size on graphs, with implications for ecology,
synthetic biology, and the modeling of processes on brain connectomes.
\end{abstract}


\section{Introduction}\label{sec:intro}

Population survival in spatially structured environments can be
formulated as a threshold problem: given local proliferation and
spatial dispersal, under which conditions does a population persist
rather than go extinct? A central question is therefore how spatial
structure and boundary effects jointly determine persistence
thresholds.

In classical reaction-diffusion models, this problem reduces to 
a spectral condition. For a species density $u:[0,\infty)\times H\rightarrow [0,1]$,
diffusing in a bounded habitat $H \subset R^d$ and evolving with a logistic
growth term according to
\begin{equation}
\partial_t u = D \nabla^2 u + ru(1-u),\qquad u|_{\partial H} = 0 
\label{eq:kiss}
\end{equation}
persistence occurs if and only if the principal Dirichlet eigenvalue
of the problem
\begin{equation}
\left(\nabla^2 u + \lambda u\right)|_H=0,\qquad  u|_{\partial H} = 0 
\label{eq:nablaspectra}
\end{equation}
satisfies $\ap{\lmin}{\partial H}\leq \rho(H)$, with $\rho(H)=r/D$,
defines the survival condition for the habitat $H$.  Note that we
treat $\rho(\cdot)$ as function of $H$ to allow for general dependency
of the parameters $r$ and $D$ on $H$. Equation~\eqref{eq:kiss} is a
simple extension of the original KiSS model, whose principal Dirichlet
eigenvalue, in a one dimensional island of size $H=[0,L]$, is given by
$\ap{\lmin}{\partial H}=\pi^2/L^2$, resulting in the well known KiSS
critical patch size $L=\pi/\sqrt{\rho}$.

Model~\eqref{eq:kiss} has been extended in multiple directions,
including metapopulation dynamics~\cite{levins1969some}, demographic
and environmental stochasticity~\cite{lande1993risks}, spatial
heterogeneity~\cite{cantrell1991effects}, and individual-based
frameworks~\cite{berti2015extinction}. However, limited attention has
been devoted to understanding how strong environmental heterogeneity
and complex connectivity structures influence population survival and
the determination of critical patch size in biological systems.
Although ecological dynamics on networks have recently received
growing attention, as in \cite{mambuca2022dynamical, lucke2023large},
to the best of our knowledge, the first attempt to discuss survival
conditions in discrete habitats, modeled as graphs, was introduced
in~\cite{tora2025critical}. However, a general spectral
characterization of survival thresholds in strongly stochastic 
network-structured habitats is still lacking.

In the following we refer to a discrete setting in which habitats are
represented as random graphs to which the spectral viewpoint discussed
above naturally extends. Let $\{H_n\}$ be a sequence of habitats
where $H_n=(G_n, S_n)$ consist of a binomial random graph $G_n$ and a
set $S_n$ of $s_n$ sink nodes. Denote by $L(G_n)$ the graph Laplacian
and by $\ap{L}{S_n}(G_n)$ the principal submatrix of $L(G_n)$ obtained by
removing rows and colums corresponding to sinks $S_n$.  In the present
work, we define the critical habitat size as the smallest real number
$\gamma$ such that, for $n\geq\gamma$, the habitat $H_n$ ensures
population survival, and we provide probabilistic bounds on it (see
Theorem~\ref{thm:appl}).  We investigate the asymptotic properties of
the inequality
\begin{equation}
\lambda_{\min}\big(\ap{L}{S_n}(G_n)\big)\leq\rho(H_n),
\end{equation} 
which, in reaction-diffusion models, characterizes population
persistence.  Some examples of deterministic and stochastic graphs
have been investigated in~\cite{tora2025critical}, but the fully
stochastic case remains largely unexplored.  We show that, in large
random habitats, population survival is asymptotically determined by
the expected number of connections between viable sites and sinks.
Formally, our main result establishes a law of large numbers for the
Dirichlet eigenvalue in the regime $s_np_n\gg \log n.$ Under this
assumption we prove that $$\lmin\big(\ap{L_n}{S_n}\big)/s_np_n
\to 1 \qquad
\mbox{almost surely}.$$ 
In particular, when $p_n\equiv p>0$ and
$s_n\geq(\log{n})^{1+\epsilon}$, the principal Dirichlet eigenvalue
concentrates around $s_np$ for any $\epsilon>0$.  
This shows that, asymptotically, the survival threshold is governed by
the expected boundary degree induced by the sinks.

\section{Survival in random graphs with sinks}\label{sec:preliminaries}
Let us briefly describe the theoretical and notational settings we work in. 
We use the symbol $\#$ to denote the cardinality of sets. As customary, 
for a positive integer $n$, we set $[n]=\{1,\ldots,n\}$. If $B$ is an order 
$n$ matrix and $S\subseteq [n]$, $\mS{B}$ is the (principal) submatrix 
obtained by deleting rows and columns whose index is in $S$. 

We consider undirected graphs on the vertex set $[n]$ and we define
$\mathcal{G}_n$ as the set of all such graphs. For a graph $G$ and
distinct vertices $i,\,j\in [n]$, we write $i\sim j$ to mean that $i$
and $j$ are adjacent in $G$. Notice that $\sim$ defines a symmetric
relation on $[n]$. The graph Laplacian, (or simply the Laplacian) 
$L(G)$ of $G$ is the self-adjoint matrix $D(G)-A(G)$ where 
$A(G)=(a_{i,j})$ is the adjacency matrix of $G$
\[a_{i,j}=\begin{cases}
	1 & \text{if $i\sim j$}\\
	0 & \text{otherwise},
\end{cases}         
\]
while $D(G)=\diago{d_G(1),\ldots,d_G(n)}$ is the degree-matrix of $G$,
where $d_G(i)=\sum_{j=1}^na_{i,j}$. If $S$ is any subset of $[n]$ and
$i$ is any vertex in $[n]\setminus S$, one has the following
decomposition
\begin{equation}\label{eq:degree_dec}
	d_G(i)=\widetilde{d}(i)+z_i(G),\quad 
\end{equation}
where $z_i(G)=\#\{j\in S: i\sim j\}$ and $\widetilde{d}_G(i)=\#\{j\in
[n]\setminus S: i\sim j\}$. Hence $z_i(G)$ is the number of neighbors of
$i$ in $S$ while $\widetilde{d}_G(i)$ is the degree of vertex $i$ in the
subgraph $\widetilde{G}$ induced by $[n]\setminus S$. Clearly $\widetilde{d}_G=d_{\widetilde{G}}$. For ease of notation, throughout the rest of the paper, when no ambiguity can arise, we will omit the explicit reference to $G$ in all the above-defined symbols. A habitat
is a pair $(G,S)$, where $G\in \mathcal{G}_n$ and $S\subseteq [n]$ is
a set of $s$ vertices thought of as sinks (or boundary points). 

Using the above definitions, the diffusion dynamics in the habitat $(G,S)$,
near the extinction threshold, can be described as
\begin{equation}\label{eq:graphkiss}
	\partial_t {\mathbf{u}} + L(G) \mathbf{u} = \rho\,\textsf{id}\mathbf{u}\qquad u_{i}|_{i\in S} = 0\,.
\end{equation}

Defining $\mS{A}(G)$ as the adjacency matrix of the subgraph
$\widetilde{G}$ of $G$ induced by $[n]\setminus S$ and $\mS{L}(G)$
the Dirichlet matrix associate to the habitat $(G,S)$, the
following identities hold
\begin{equation}\label{eq:identity}
\mS{L}(G)=\mS{D}(G)-\mS{A}(G)=Z+\mS{\widetilde{L}}(G)
\end{equation}
where $Z$ is the diagonal matrix whose diagonal elements are the
$z_i$'s, $i\in [n]\setminus S$, defined in \eqref{eq:degree_dec} while $\mS{\widetilde{L}}(G)$
is the Laplacian of $\widetilde{G}$. The above expression implies
that, in general, $\mS{L}(G)$ is not itself a graph Laplacian.

Dynamics \eqref{eq:graphkiss} can be written as
\begin{equation}\label{eq:graphkissNOSINK}
	\partial_t {\mathbf{u}} + \mS{L}(G) \mathbf{u} = \rho\,\textsf{id}\mathbf{u}
\end{equation}
and, considering that Dirichlet eigenvalue of $(G,S)$, denoted by
$\ap{\xi}{S}(G)=\lambda_{\min}\big(\ap{L}{S}(G)\big)$, i.e. the smallest eigenvalue of $\mS{L}(G)$,
survival condition for \eqref{eq:graphkiss} and \eqref{eq:graphkiss} is
\begin{equation}\label{eq:survivalConditions}
\ap{\xi}{S}(G)\leq\rho.
\end{equation} 

\subsection{Recall on Chernoff Bound}
A binomial random graph $G$ on $[n]$ is a sample point from 
the probability space $\G(n,p_n)=(\mathcal{G}_n,\Pm_{n,p_n})$, where
\begin{equation*}\label{eq:binp}
	\Pm_{n,p_n}(G)=p_n^{\size{G}}(1-p_n)^{{n \choose 2}-\size{G}},
\end{equation*}
where $\size{G}$ is the number of edges of $G$. More pictorially, a binomial random graph is a graph on $[n]$ where each of the ${n \choose 2}$ pairs of different vertices is joined by an edge with the same probability $p_n$, independently of each other pair. We also stipulate that the sinks of a random habitat are always located in the first $s_n$ vertices of the underlying random graph and we call the corresponding habitat, the standard random habitat. This assumption causes no loss of generality because if $(G,S)$ is any habitat where $G\in\mathcal{G}_n$ and $S$ has $s$ sinks, then there is a graph isomorphism $\pi:[n]\rightarrow [n]$ such that $\pi(S)=[s_n]$. Therefore, the spectra of the matrices $\mS{L}$ and $\ap{L_n}{[s_n]}$ are the same, as they are conjugate by a permutation matrix. From now on, we abridge the more precise, albeit cumbersome and redundant, expression $\ap{L_n}{[s_n]}$ by $\ma{L}$. We tacitly apply the same convention to other related notations. In particular, hereafter $\ma{\xi}$ will denote the Dirichlet eigenvalue of a standard random habitat with $s_n$ sinks and $n-s_n$ non-sink vertices.      

The sequences of random objects (random variables, matrices) we are interested in are defined on different probability spaces for different $n$. Using the standard procedure of taking product of probability spaces, we can ensure that all the objects are defined on the same probability space, $(\Omega,\F,\Pm)$. Our forthcoming convergence results are always stated with respect to this probability space. Also, for a sequence of random variable $\{U_n\}$ and a random variable $U$, all being defined on $(\Omega,\F,\Pm)$, we use that shorthand notation $U_n\xrightarrow{a.s.} U$ to mean that $U_n$ converge to $U$ almost surely. We will make an extensive use of the following variants of the so-called Chernoff bounds (see~\cite{frieze2015introduction}, Chapter 27) that for a random variable $X\sim\bin{\nu}{p}$ and $\epsilon>0$ read as 
\begin{alignat*}{3}
	\Pm\left(X \geq (1+\epsilon)\nu p\right)&\leq \phantom{.}e^{-\epsilon^2\nu p/3}&\phantom{-} &\text{Upper tail Chernoff Bound,}\\
	\Pm\left(X \leq (1-\epsilon)\nu p\right)&\leq \phantom{.}e^{-\epsilon^2\nu p/2} &\phantom{-}&\text{Lower tail Chernoff Bound,}\\
	\Pm\left(|X - \mu| \geq \epsilon\nu p\right)&\leq 2e^{-\epsilon^2\nu p/3} &\phantom{-}&\text{Symmetric Chernoff Bound}.
\end{alignat*}

\section{Results}\label{sec:results}
In the present work, we extend this concept to that of critical
habitat size, defined as the smallest real number $\gamma$ such that,
for $n\geq\gamma$, the habitat $H_n$ ensures population survival, and
we provide probabilistic bounds on it (see Theorem~\ref{thm:appl}).
Actually the main aim of this paper is to study inequality 
\begin{equation}\label{eq:fund_in}
\ap{\xi}{S_n}(G_n)\leq\rho(H_n),
\end{equation} 
for a sequence of habitats $\{H_n\}$, $H_n=(G_n,S_n)$, where $G_n$ is a binomial random graph on $\{1,2,\ldots,n\}$, i.e., a graph sampled from the probability space $\G(n,p_n)$ (see Section \ref{sec:results} for precise definition), $S_n$ consists of $s_n$ sinks, the sequence $\{s_n\}$ does not decrease and $\ap{L}{S_n}(G_n)$ is the principal submatrix of the graph Laplacian of $G_n$ induced by the non-sink vertices. By what we have observed and stipulated above, there is no loss of generality in assuming that $H_n$ is a standard random habitat. Consequently, \eqref{eq:fund_in} is equivalent to   
\begin{equation}\label{eq:fund_in_1}
	\ma{\xi}\leq\rho(H_n),
\end{equation}
According to whether it satisfies the survival condition or not, a habitat will be referred to as healthy or deadly. We establish law of large numbers for the Dirichlet eigenvalue of $H_n$ under the assumptions that $s_np_n$ grows faster than logarithmically. In particular, we show that $\ma{\xi}/s_np_n$ converges to 1 almost surely. This happens, for instance, when $p_n$ is a fixed constant and $s_n\geq(\log{n})^{1+\epsilon}$ for any $\epsilon>0$. Specifically, the following result holds, where, for functions $f,\,g: \mathbb{N}\rightarrow \mathbb{N}$, we write $f(n)\ll g(n)$ to mean that $\lim_{n\rightarrow \infty}\frac{f(n)}{g(n)}=0$.

\begin{theorem}\label{thm:main00}
Consider a sequence $\{H_n\}$ of random habitats of the form $(G_n,S_n)$ where $G_n$ is sampled from $\G(n,p_n)$, $S_n$ has $s_n$ sinks and the sequence $\{s_n\}$ is non-decreasing. Let $\big\{\ap{\xi_n}{s_n}\big\}$ be the corresponding sequence of Dirichlet eigenvalues. 
\begin{enumerate}[label={\rm (\roman*)}]
\item\label{com:0i} If $s_n\ll n$, then the sequence $\{p^*_n\}$, $p^*_n=\frac{\log{n}}{n}$, is a sharp threshold for the positiveness of the Dirichlet eigenvalue of $H_n$, i.e.,
$$\lim_{n\rightarrow\infty}\Pm\big(\ap{\xi_n}{s_n}>0\big)=\begin{cases}
	0 & \text{if $\phantom{-}\frac{p_n}{p^*_n}\leq 1-\epsilon$}\\
	1 & \text{if $\phantom{-}\frac{p_n}{p^*_n}\geq 1+\epsilon$}.\\
\end{cases}  
$$
for every $\epsilon>0$.
\item\label{com:0ii} Consider the following assumptions
	\begin{enumerate}[label={\rm (\alph*)}]
		\item\label{com:a} $\log{n}\ll(n-s_n)p_n$.
		\item\label{com:b} $\lim_{n\rightarrow \infty}s_n/n<1$ and $\log{n}\ll s_np_n$.   
	\end{enumerate}
If assumption \ref{com:b} holds, then $\mydfrac{\ap{\xi_n}{s_n}}{s_np_n}$ converges to $1$ almost surely. Assumptions \ref{com:a} and \ref{com:b} both imply that $\Pm\big(\limsup_n\left\{\ap{\xi_n}{s_n}/s_np_n\right\}\leq 1\big)=1$. 
\end{enumerate}
As a consequence, if assumption \ref{com:b} holds, then the sequence $\{\rho^*_n\}$, $\rho^*_n=s_np_n$, is a sharp threshold for habitat's viability ruled by the survival condition \eqref{eq:fund_in}. Under assumption \ref{com:a}, for every $\epsilon>0$, if $\rho(H_n)\geq (1+\epsilon)s_np_n$, then for sufficiently large $n$, almost every habitat with $n$ vertices is healthy. This happens, in particular, if $s_np_n$ converges to zero and $\rho(H_n)$ is bounded away from zero.  
\end{theorem}
\begin{figure}[!h]
	\includegraphics[scale=0.90]{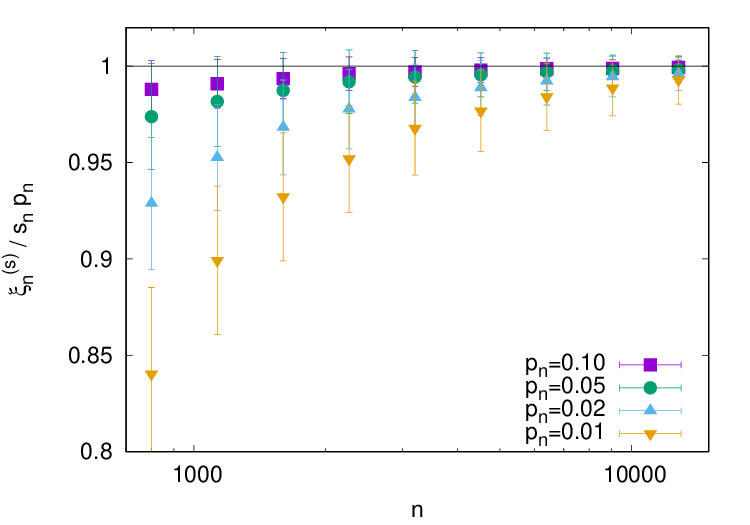}
	\caption{Numerical computation of $\ma{\xi}/s_np_n$ vs $n$ for
	fixed $s_n=50$ and different values of $p_n$.  The plot shows
	the mean and the standard deviation of $\ma{\xi}/s_np_n$
	computed over 1000 random habitats.}  \label{fig:1}
\end{figure}
If the sequence $\{p_n\}$ is such that, eventually, $p_n\geq
C\log{n}/n$ for some constant $C>1$, then it is customary to say that
random graph evolves in the connected regime while, if
$np_n\gg\log{n}$, then we may say that the random graph evolves in the
quasi-dense regime. Consequently, statement \ref{com:0i} in
Theorem \ref{thm:main00} establishes that a non-zero limit for
$\ap{\xi_n}{s_n}$ can exist only in the connected regime while the
subsequent statement \ref{com:0ii} identifies this limit in the
quasi-dense regime under assumption \ref{com:b}. Note that assumptions
\ref{com:a} and \ref{com:b} are independent of each other and both
imply that the random graph evolves in the quasi-dense regime. In such
a regime, the condition $s_np_n\rightarrow 0$, occurring in the last
part of Theorem \ref{thm:main00}, can be realized with
$p_n=\log^a{n}/n$, for some $a>1$, and with $s_n$ either constant or
growing with $n$. Actually, as explained in Remark~\ref{rem:cavolo},
it is not difficult to prove that, almost surely, for $n$ large
enough, $s_np_n-\kappa(n)\leq \ap{\xi_n}{s_n}\leq s_np_n$, where
$k(n)$ is $o(np_n)$ (see Figure \ref{fig:1}). Also, it is worth
observing, that $s_np_n$ is the smallest eigenvalue of the expectation
of the random matrix $\ap{L_n}{S_n}$ rather then than expected value
of the random variable
$\ap{\xi_n}{s_n}=\lmin\big(\ap{L_n}{S_n}\big)$. As explained in
Remark~\ref{rem:expected}, one has $\E\left(\ap{\xi_n}{s_n}\right)\leq
s_np_n$.

We warn the reader that, by combining Oliveira’s matrix concentration for random Laplacians (see \cite{oliveira2009concentration}) with the aforementioned identity $s_np_n=\lmin\big(\E\big(\ma{L}\big)\big)$ (see Remark \ref{rem:expected}), one can prove the same asymptotic behaviour of the Dirichlet eigenvalue in the dense regime. In fact, Theorem \ref{thm:main00}.\ref{com:0ii} can be recovered from these general results combined with the deterministic bounds of Lemma \ref{lem:easy}. Since this route requires substantially heavier machinery and does not simplify the argument, for the sake of clarity and self‑containment, we prefer to derive Theorem \ref{thm:main00}.\ref{com:0ii} directly from elementary first principles, using only Chernoff bounds and the Borel–Cantelli lemma. While this avoids resorting the machinery of the general theory of random graphs, it highlights the essential mechanism behind the behaviour of  $\ma{\xi}$.

Using the fact that the survival condition defines a monotone graph
property (see the proof of Theorem \ref{thm:appl}), we obtain
probabilistic bounds on the minimum size of healthy habitats, i.e.,
the critical habitat size. Specifically, when habitats are sampled
uniformly at random from the set of all possible habitats with a given
number of non-sink vertices and, possibly, a given number of edges, we
have the following `probabilistic solution' to the critical habitat size
problem in model \eqref{eq:graphkissNOSINK}:
\begin{theorem}\label{thm:appl}
Let $s$ be a positive integer number, $\rho$ a positive real number 
and $\delta$ be a real number arbitrarily picked in $(0,1)$.
\begin{enumerate}[label={\rm (\roman*)}]
	\item\label{com:ia} For any positive real number $\epsilon$ let $p_\epsilon:=p_\epsilon(\rho,s)$ where $p_\epsilon(\rho,s)=\frac{\rho}{(1+\epsilon)s}$. Set $\mu=sp_\epsilon$. If $n$ and $m$ are positive integer numbers such 
		\begin{equation}\label{eq:cps}
		n\geq s+\mydfrac{3\mu}{(\rho-\mu)^2}\log{\frac{4}{\delta}}\quad\text{and}\quad m\leq p_\epsilon {n \choose 2},
	\end{equation}
	then all but at most fraction $\delta$ of the habitats with $s$ sinks whose underlying graph has $n$ vertices and $m$ edges are healthy 
	\item\label{com:iia} Except possibly for a fraction $\delta$, if $\rho\geq s/2$, then all habitats with $s$ sinks and at least $\frac{6s}{(s-2\rho)^2}\log{\frac{1}{\delta}}$ non-sink vertices are healthy, while, if $\rho\leq s/2$ all habitats with $s$ sinks and no more than $\delta e^{(s-2\rho)^2/4s}$ non-sink vertices are deadly.
\end{enumerate}
\end{theorem}
\begin{figure}[!h]
	\includegraphics[scale=0.80]{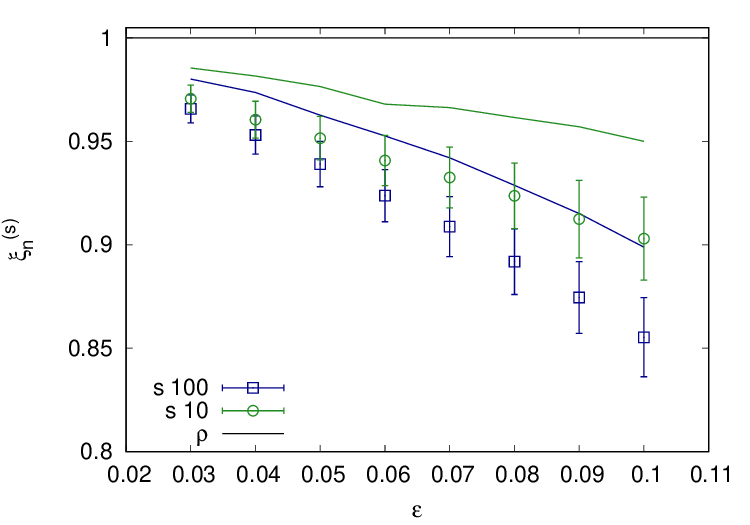}
	\caption{Numerical computation of $\ma{\xi}$ vs $\epsilon$ for
	$\delta=10^{-2}$, $\rho=1$, and $s=10$ (dark green) and $s=100$
	(dark blue).  The plot shows the mean and the standard
	deviation of $\ma{\xi}$ computed over 1000 random habitats
	chosen according to \eqref{eq:cps} (points with error bars),
	along with the $0.01$ right-tail percentile of the sample
	distrution (full lines with the corresponding colors).}  \label{fig:2}
\end{figure}
Theorem \ref{thm:appl} provides a static counterpart to Theorem \ref{thm:main00}, in that it quantifies the typical health of a uniformly sampled habitat of size $n$ instead of describing the behavior of $\xi^{S_n}(G_n)$ along a random growth process. Besides the fact that that the survival condition defines a monotone graph property, the proof of the theorem uses only the standard comparison between $\G(n,p_n)$ and $\G_{n,m}$. Yet, the result offers a clean ecological interpretation, stating that sufficiently large habitats are almost surely healthy, while sufficiently small ones are almost surely deadly.

For instance, when $\delta=10^{-2}$, $\rho=1$ and $s=10$, according to
Theorem~\ref{thm:appl}.\ref{com:ia}, for $\epsilon=0.1$ at least 99\%
of the habitats with at least 1988 vertices are healthy, while for
$\epsilon=0.03$ the same fraction of the habitats is healthy with a
much higher number of vertices, at least 20581.  However, as shown in
Fig.~\ref{fig:2}, such bounds become sharp only in the asymptotic
regime of very large $n$ (i.e. for small $\epsilon$), not easily
accessible to numerical simulation.

To the best of our knowledge, this paper is the first systematic
attempt to give a general framework to address the critical patch 
size problem in graphs.

\section{Proofs}
We preliminary need a couples of lemmas. 
\begin{lemma}\label{lem:easy}
	Let $(G,S)$ be a habitat where $G$ has $n$ vertices and $S$ has $s$ sinks. For $i\not\in S$, let $z_i$ and $\widetilde{d}(i)$ be defined as in \eqref{eq:degree_dec}. Let $\mS{\delta}(G)$ be the number of edges of $G$ with exactly one end-vertex in $S$. Set $\mS{\zeta}=\min_{i\not\in S}z_i$ and $\mS{\vartheta}(G)=\mS{\delta}(G)/(n-s)$. Then, the Dirichlet eigenvalue $\mS{\xi}$ of $(G,S)$ satisfies the inequalities $\mS{\zeta}\leq\mS{\xi}(G)\leq \mS{\vartheta}(G)$ and both the inequalities are sharp. 	Furthermore, if $G$ is connected, then $\mS{L}(G)$ is positive definite. As a consequence, if $\ma{\zeta}$ and $\ma{\vartheta}$ are the random variable described by $\mS{\zeta(G)}$ and $\mS{\vartheta(G)}$, respectively, when $(G,S)$ is sampled from a standard random habitat with $s_n$ sinks, then the bounds stated above hold almost surely under $\Pm$ with the latter symbols replaced by the former ones.
\end{lemma}
\begin{proof} The Laplacian $L(G)$ of any graph $G$, is always 
positive semi-definite and has not full rank
(\cite{royle2001algebraic}): if $\boldsymbol{e}$ is the all ones
vector of appropriate dimension, then
$L(G)\boldsymbol{e}=\boldsymbol{0}$. For ease of notation, throughout
the proof we omit the reference to $G$. The identities
$\lmin\big(\mS{L}\big)=\lmin\big(\mS{D}-\mS{A}\big)=\lmin\big(Z+\mS{\widetilde{L}}\big)$
follows by \eqref{eq:identity} where all of the symbols retain the
meaning there defined. Therefore
	\[
	\lmin\big(\mS{L}\big)\geq \lmin\big(Z\big)+\lmin\big(\mS{\widetilde{L}}\big)= \lmin\big(Z\big)
	\]
	where the last equality holds because $\mS{\widetilde{L}}$ has not full rank. The lower bound follows once we observe that $\lmin(Z)=\min_{i\in [n]\setminus S}z_i=\mS{\zeta}$. To prove the sharpness of the bound $\mS{\zeta}\leq\mS{\xi}(G)$, it suffices to exhibit a graph whose Dirichlet eigenvalues attains equality. Let $G$ be a graph such that the vertex $\kappa$ attaining the minimum of $z_i$ over $[n]\setminus S$ is isolated in $\widetilde{G}$. Hence $d_{\widetilde{G}}(\kappa)=0$ and $d_G(\kappa)=z_\kappa=\mS{\zeta}$. Let $\boldsymbol{e}_{\kappa}$ be the ${\kappa}$-th fundamental basis vector defined by  $\boldsymbol{e}_{\kappa}(i)=1$ if $i=\kappa$ and $\boldsymbol{e}_{\kappa}(i)=0$ if $i\not=0$.
	Observe that 
	$$\boldsymbol{e}_\kappa^\top\mS{\widetilde{L}}\boldsymbol{e}_\kappa=d_{\widetilde{G}}(\kappa)=0.$$
	Hence
	$$\boldsymbol{e}_\kappa^\top\mS{L}\boldsymbol{e}_\kappa=\boldsymbol{e}_\kappa^\top \left(Z+\mS{\widetilde{L}}\right)\boldsymbol{e}_\kappa=\boldsymbol{e}_\kappa^\top Z{e}_\kappa=z_\kappa= \mS{\zeta}.$$
	Since in any real symmetric matrix $B$, the smallest eigenvalue of $B$ is the minimum value of the Rayleigh quotient of $B$, namely the minimum of the quadratic form $Q(x)=x^\top B x$ over the sphere $\{\|x\|=1\}$, it follows that
	\[
	\mS{\zeta}\leq \mS{\xi}(G)=\min_{u\not=0}\mydfrac{u^\top \mS{L} u}{\|u\|^2}\leq \boldsymbol{e}_\kappa^\top\mS{L}\boldsymbol{e}_\kappa=\mS{\zeta}
	\] 
	which proves that, for the graph $G$ we have constructed, $\mS{\zeta}=\mS{\xi}(G)$ as required.
	
	To establish the upper bound, observe that $\mS{\delta}=\sum_{i\in [n]\setminus S}z_i$. Indeed, an edge of $G$ contributes to the sum if and only it has exactly one end-vertex in $[n]\setminus S$ and, hence, exactly one end-vertex in $S$. Therefore 
	$$\mS{\delta}=\sum_{i\in [n]\setminus S}z_i=\boldsymbol{e}^\top Z\boldsymbol{e}=\boldsymbol{e}^\top\left(\mS{L}-\mS{\widetilde{L}}\right)\boldsymbol{e}=\boldsymbol{e}^\top \mS{L}\boldsymbol{e}.$$
	Therefore, after observing that $\|\boldsymbol{e}\|^2=n-s$, it follows that
	\[
	\begin{split}
		\lmin(\mS{L})=\min_{u\not=0}\mydfrac{u^\top \mS{L} u}{\|u\|^2}\leq \mydfrac{\boldsymbol{e}^\top \mS{L}\boldsymbol{e}}{\|\boldsymbol{e}\|^2}=\mydfrac{\mS{\delta}}{n-s}
	\end{split} 	
	\]
	as claimed. As above, to prove the sharpness of the bound $\mS{\xi}(G)\leq \mS{\vartheta}(G)$, it suffices to exhibit a graph whose Dirichlet eigenvalues attains equality. Let $G$ be a graph such that, for some positive integer $\theta$ such that $\theta\leq s$, $z_i=\theta$ for $i\in [n]\setminus S$. It is clear that such a graph exists: for each vertex in $i\in [n]\setminus S$ connect $i$ to a arbitrarily chosen subset of $S$ of cardinality $\theta$. For such a graph $\mS{\vartheta}(G)=\theta$. On the other hand, $\mS{L}(G)=\theta I+\mS{\widetilde{L}}(G)$ where $I$ is the identity matrix of order $n-s$. Since $\mS{L}(S)$ is a diagonal update of $\mS{\widetilde{L}}(S)$ by a scalar matrix, the spectrum of $\mS{L}$ is shifted by $a$ with respect to the spectrum of $\mS{\widetilde{L}}$. Since $\mS{\widetilde{L}}$ is a Laplacian matrix, its minimum eigenvalue is zero. Therefore the minimum eigenvalue of $\mS{L}$ is $\theta$ so that,
	$$\theta=\mS{\xi}(G)\leq \mS{\vartheta}(G)=\theta$$
	which proves that, for the graph $G$ we have constructed, $\mS{\xi}(G)=\mS{\vartheta}(G)$ as required.
	To prove the last part of the lemma, it suffices to observe that if $G$ is connected, then $\mS{L}$ is an irreducible diagonally dominant matrix (see \cite{varga1962iterative} for this notion) and that, as consequence of Gershgorin circle theorem, every symmetric and irreducible diagonally dominant matrix is positive definite (Corollary to Theorem 1.8 in \cite{varga1962iterative}).
\end{proof}
\begin{remark}\label{rem:fund}
	As innocent and folklore as it may seem, Lemma \ref{lem:easy} provides the optimal deterministic range for the Dirichlet eigenvalue in a precise sense. Without imposing further structural assumptions, such as conditions on the internal graph $\widetilde{G}$, the bounds $\mS{\zeta}\leq\mS{\xi}(G)\leq \mS{\vartheta}(G)$ cannot be improved: both extremes can be attained by suitable families of graphs---via full localization of the Rayleigh minimizer on a single vertex of $[n]\setminus S$ or by taking a graph with constant degree sequence on $[n]\setminus S$---and no sharper universal estimate is possible without loosing the generality of the random model we have chosen.
\end{remark}

\begin{lemma}\label{lem:degree_as} 
	If assumptions \ref{com:a} or \ref{com:b} hold, then $\ma{\vartheta}/s_np_n\xrightarrow{a.s.} 1$. If assumption \ref{com:b} holds, then $\ma{\zeta}/s_np_n\xrightarrow{a.s.} 1$.  
\end{lemma}
\begin{proof} For ease of notation, let $s=s_n$, $p=p_n$ and $\nn=n-s$. For random variables $X$ and $Y$ write $X\sim \ber(p)$ to mean that $X$ is a Bernoulli random variable with mean $p$ and $Y\sim \bin{n}{p}$ to mean that $Y$ has a the binomial distribution with parameters $n$ ad $p$, which is the distribution of the sum of $n$ independent Bernoulli random variables having the same mean $p$. Let $\ap{e_{i,j}}{n}$ be the indicator of the event that vertices $i$ and $j$ are adjacent in the random graph. Hence $\ap{e_{i,j}}{n}\sim\ber(p)$ and $\nn\ma{\vartheta}=\sum_i^s\sum_{j=s+1}^n\ap{e_{i,j}}{n}$ with the $\ap{e_{i,j}}{n}$ being independent. Therefore $\nn\ma{\vartheta}\sim \bin{s\nn}{p}$ under $\Pm$. 
	For $i=s+1,\ldots, n$, let $\ap{Z_{i,n}}{s}$ be defined by $\sum_{j=1}^s\ap{e_{i,j}}{n}$. Hence $\ma{\zeta}=\min_{i=s+1,\ldots,n}\ap{Z_{i,n}}{s}$ where $\ap{Z_{i,n}}{s}\sim \bin{s}{p}$. Let 
	$$B_{1,n}(\epsilon)=\big\{\omega\in\Omega : |\ma{\vartheta}(\omega)-sp|\geq \epsilon sp\big\}\,\text{and}\,B_{2,n}(\epsilon)=\big\{\omega\in\Omega : |\ma{\zeta}(\omega)-(n-1)p|\geq \epsilon sp\big\}.$$
	By a standard application of the Borel-Cantelli Lemma (see, e.g., \cite{bremaud2024introduction}, Chapter 6), if, for every $\epsilon>0$, we can prove that for $h=1$ and $h=2$, $\sum_{n\geq 1}\Pm\left(B_{h,n}(\epsilon)\right)$ if finite, under, respectively, assumptions \ref{com:a} or \ref{com:b}, and \ref{com:b}, then we have proved the lemma. To prove that $\sum_{n\geq 1}\Pm\left(B_{h,n}(\epsilon)\right)$ is finite for $h=1,2$, we show that $\Pm\left(B_{h,n}(\epsilon)\right)$ is $O(n^{-2})$. An application of the symmetric Chernoff bound to $\nn\ma{\vartheta}$ yields
	\[
	\Pm\big(B_{1,n}(\epsilon)\big)=\Pm\big(|\ma{\vartheta}(\omega)-sp|\geq \epsilon sp\big)=\Pm\big(|\nn\ma{\vartheta}(\omega)-\nn sp|\geq \epsilon \nn sp\big)\leq 2e^{-\epsilon^2\nn s p/3}.
	\]
	Assumptions \ref{com:a} and \ref{com:b} both imply that $\nn s p\gg \log{n}$. Hence, for $n$ sufficiently large, it holds that $\nn s p\geq \frac{6}{\epsilon^2}\log{n}$ implying that $\Pm\left(B_{1,n}(\epsilon)\right)\leq 2e^{-\epsilon^2\nn s p/3}\leq 2e^{-2\log{n}}=2/n^{2}$. This establishes the first almost sure limit. To establish the other one, observe that $B_{2,n}(\epsilon)\subseteq \cup_{s+1}^nF_i(\epsilon)$ where, for $i=s+1,\ldots,n$, $F_i(\epsilon):=\big\{\omega\in\Omega : |\ap{Z_{i,n}}{s}(\omega)-sp|\geq \epsilon s p\big\}$.
	Therefore, since $\ap{Z_{i,n}}{s}$ is distributed as $Z$ where $Z\sim\bin{s}{p}$, it follows that
	$\Pm\left(B_{2,n}(\epsilon)\right)\leq \nn\Pm\left(|Z-s p|\geq \epsilon s p\right)$, 
	where the factor $\nn$ is due to the Union Bound. An application of the symmetric Chernoff bound to $Z$ yields 
	$$\Pm\left(B_{2,n}(\epsilon)\right)\leq \nn\Pm\left(|Z-s p|\geq \epsilon s p\right) \leq 2\nn e^{-\epsilon^2 s p/3}.$$
	By assumption \ref{com:b}, if $n$ is sufficiently large, then $sp\geq \frac{9}{\epsilon^2}\log{n}$. Hence $\Pm\left(B_{3,n}(\epsilon)\right)\leq 2\nn e^{-3\log{n}}\leq 2n e^{-3\log{n}}\leq 2/n^{2}$. This finishes the proof of the lemma. 
\end{proof}
\begin{theoremproof}{\ref{thm:main00}}
	Let us prove statement \ref{com:0i} first. Let $F_n$ be the event that in $H_n$ the the underlying random graph is connected. By Lemma \ref{lem:easy}, event $F_n$ implies the event that $\ma{L}$ is positive definite. Hence $\Pm(F_n)\leq \Pm(\ma{\xi}>0)\leq 1$. Since in the connected regime $\lim_{n\rightarrow\infty}\Pm(F_n)=1$ (see, e.g., \cite{frieze2015introduction}), we conclude $\lim_{n\rightarrow\infty}\Pm(\ma{\xi}>0)=1$ as well. This proves the only if part. Conversely, if $p_n$ is below the threshold, then consider the sequence formed by the random graphs induced by the non-sink vertices. Since the event that the graphs in the sequence contain a isolated vertex is an asymptotically almost sure event (see, e.g., \cite{janson2011random}), it follows that the event that $\ma{L}$ has an all zero row is also asymptotically almost sure and so is the event that $\ma{L}$ is singular and the proof of the statement is complete. 
	
	As for statement \ref{com:0ii}, observe that by Lemma \ref{lem:easy}, it almost surely holds that 
	$\ma{\zeta}\leq \ma{\xi}\leq\ma{\vartheta}$. On the other hand, by Lemma \ref{lem:degree_as}, both $\ma{\zeta}/sp_n$ and $\ma{\vartheta}/sp_n$ converge to $1$ almost surely under assumption \ref{com:b}. Hence $\ma{\xi}/sp_n\xrightarrow{a.s.}1$ follows by squeezing. The second part follows straightforwardly from the almost sure inequality $\ma{\xi}\leq \ma{\vartheta}$ and the almost sure limit $\ma{\vartheta}/s_np_n\xrightarrow{a.s.}1$ under assumptions \ref{com:a} or \ref{com:b}. These two facts imply
	\[
	\begin{split}
		1\geq\Pm\left(\limsup_n\mydfrac{\ma{\xi}}{s_np_n}\leq 1\right)\geq \Pm\left(\limsup_n\mydfrac{\ma{\vartheta}}{s_np_n}\leq 1\right)=\Pm\left(\lim_n\mydfrac{\ma{\vartheta}}{s_np_n}\leq 1\right)=1.
	\end{split}
	\]
	The proof of the theorem is now complete.
\end{theoremproof}
\begin{remark}\label{rem:cavolo}
	For a standard random habitat, let $\ap{d_{n,\min}}{s}$ be the minimum among the degrees of the non-sink vertices. Since $\lmin\big(\ma{D}-\ma{A}\big)\geq \ap{d_{n,\min}}{s}-\lmax\big(\ma{A}\big)$, it almost surely holds that $\ma{\xi}\geq \ap{d_{n,\min}}{s}-\lmax\big(\ma{A}\big)$. Since, reasoning as in Lemma \ref{lem:degree_as}, it can be shown $\ap{d_{n,\min}}{s}=(1+o(1))(n-1)p_n$ and since $\lmax\big(\ma{A}\big)=(1+o(1))(n-s_n-1)p_n$ (see \cite{Chungetal}), it follows that $\ma{\xi}=s_np_n-\kappa(n)$ where  $\kappa(n)\ll np_n$.
\end{remark}

\begin{remark}\label{rem:expected}
	By the results in \cite{oliveira2009concentration}, it can be shown that $s_np_n=\lmin\big(\E\big(\ma{L}\big)\big)$, where $\E\big(\ma{L}\big)$ is the expectation under $\Pm$ of the random matrix $\ma{L}$. The map which takes a symmetric matrix to its smallest eigenvalue is a concave map on the set of real symmetric matrix. Hence, by Jensen inequality, $\E\big(\ma{\xi}\big)\leq s_np_n$.  
\end{remark}

\begin{theoremproof}{\ref{thm:appl}}
Let $n$ and $m$ be positive integers and $\GG_{n,m}$ be the set of graphs on $[n]$ with $m$ edges. Let $\mt{\mathcal{H}}{s}{n,m}$ be the set of habitats $(G,S)$ such that $G\in\GG_{n,m}$ and $S$ consists of $s$ sinks. Likewise, $\mt{\mathcal{H}}{s}{n}$ is the set of habitats $(G,S)$ such that $G\in\GG_n$ and $S$ consists of $s$ sinks. Also recall that a standard habitat is a habitat $(G,[s])$. Let $\ap{\mathcal{H}}{s}(S)$ be the subset of $\mt{\mathcal{H}}{s}{n,m}$ ($\mt{\mathcal{H}}{s}{n}$, resp.) consisting of the habitats whose set of sinks is $S$. It is clear that $\ap{\mathcal{H}}{s}(S)$ and $\ap{\mathcal{H}}{s}([s])$ have the same number of elements. Denote by $\ap{L}{s}(G)$ the principal submatrix of $L(G)$ induced by $\{s+1,\ldots,n\}$. 
\begin{claimproof}{\ref{com:ia}}
	Let $\mt{\Pm}{s}{n,m}$ be the uniform measure on $\mt{\mathcal{H}}{s}{n,m}$ and and let  $B(S)$ the event which occurs when a habitat sampled uniformly at random from $\mt{\mathcal{H}}{s}{n,m}$ is healthy and its set of sinks is precisely $S$. Also, let $B$ the event that a habitat sampled uniformly at random from $\mt{\mathcal{H}}{s}{n,m}$ is healthy. The thesis of the statement is then equivalent to $\mt{\Pm}{s}{n,m}(B)\geq 1-\delta$. By symmetry of the uniform model $\G_{n,m}$ under vertex permutations, the probability that a habitat $(G,S)$ is healthy depends only on $|S|=s$ and not on the specific choice of $S$. Indeed, for any two sets $S,\,S'\subseteq [n]$ with $|S|=|S'|=s$, there is a permutation sending $S$ to $S'$ that takes $L^{(S)}(G)$ to a similar matrix $L^{(S')}(G')$, hence preserving the probability of $B(S)$. Thus all events $B(S)$ have the same probability. In particular, for $S=[s]$, it holds that $\mt{\Pm}{s}{n,m}(B)=\mt{\Pm}{s}{n,m}(B([s]))=\Pm_{n,m}\big(\ap{L}{s}(G)\leq \rho\big)$,
	so that, to prove the theorem, it suffices to prove that
	\begin{equation}\label{eq:lastfund}
		\Pm_{n,m}\big(\ap{L}{s}(G)\leq \rho\big)\geq 1-\delta.
	\end{equation}
	Since, as we shall show below, the event $\big\{G\in \GG_n \ |\ \lmin\big(\ap{L(G)}{s}\big)\geq \rho\big\}$ defines a monotone increasing graph property $\prd^*_n$, we may apply the standard comparison between $\G(n,p_n)$ and $\G_{n,m}$ (see e.g.. Lemma 1.3 in \cite{frieze2015introduction} or Proposition 1.12 in \cite{janson2011random}). In particular, with $p=m/{n\choose 2}$, it holds that 
	\begin{equation}\label{eq:llfund}
	\Pm_{n,m}\left(G\in \prd_n^*\right)\leq 4\Pm_{n,p}\left(G\in \prd_n^*\right)\leq 4\Pm_{n,p_\epsilon}\left(G\in \prd_n^*\right)	
\end{equation}
where the constant $4$ comes from the lower bound $\Pm(X\geq m)\geq 1/4$ valid for a binomial random variable $X$  whenever $m\leq \Ea{X}$ (see \cite{greenberg2014tight}) while the last inequality is due to the hypothesis $m\leq p_\epsilon {n\choose 2}$ upon recalling that $\Pm_{n,p}\big(G\in \prd_n^*\big)$ is not decreasing in $p$ (see again \cite{frieze2015introduction,janson2011random}). 

Consider now the property $\prd^*_n=\big\{G\in \GG_n \ |\ \lmin\big(\ap{L(G)}{s}\big)\geq \rho\big\}$. We claim that $\prd^*_n$ is a monotone graph property. It suffices to show that if $G'$ has one edge more than $G$, then $\lmin\big(\ap{L}{s}(G')\big)\geq\lmin\big(\ap{L}{s}(G)\big)$. Let $e$ be the extra edge in $G'$. Suppose that $e$ connects the distinct vertices $i$ and $j$. If $\{i,j\}\subseteq\{1,\ldots,s\}$, then $\lmin\big(\ap{L}{s}(G')\big)=\lmin\big(\ap{L}{s}(G)\big)$ because $\ap{L}{s}(G')=\ap{L}{s}(G)$ in this case. If either $i$ or $j$ belongs to $\{1,\ldots,s\}$, then $\lmin\big(\ap{L}{s}(G')\big)=\lmin\big(\ap{L}{s}(G)+\ell_\kappa\big)$, where $\ell_\kappa$ is an order $(n-s)$ diagonal matrix all whose entries are zero except for the $\kappa$-th diagonal one which is 1, where $\kappa\in\{i,j\}$. If $\{i,j\}\subseteq [n]\setminus \{1,\ldots,s\}$, then $\lmin\big(\ap{L}{s}(G')\big)=\lmin\big(\ap{L}{s}(G)+\ell_{i,j}\big)$, where $\ell_{i,j}$ is the graph Laplacian of the graph $(\{s+1,\ldots,n\},\{e\})$. Since both $\ell_\kappa$ and $\ell_{i,j}$ are semi-definite positive matrices, in any case it holds that $\lmin\big(\ap{L}{s}(G')\big)\geq\lmin\big(\ap{L}{s}(G)\big)$. Hence $\prd^*_n$ is monotone as claimed. Recall now that by Lemma \ref{lem:degree_as}, $\nn\ma{\vartheta}\sim \bin{\nn s}{p_\epsilon}$, where $\tilde{n}=n-s$, and by Lemma \ref{lem:easy}, it almost surely holds that $\ma{\xi} \leq \ma{\vartheta}$. Hence
	\[
	\begin{split}
	\Pm_{n,p_\epsilon}\left(G\in \prd^*_n\right)&=\Pm_{n,p_\epsilon}\left(\ma{\xi}\geq \rho\right)=\Pm_{n,p_\epsilon}\left(\ma{\xi}\geq (1+\epsilon)\mu\right)\\
		&\leq \Pm_{n,p_\epsilon}(\nn\ma{\vartheta}\geq (1+\epsilon)\mu \tilde{n})\leq \exp\left(-\mydfrac{\epsilon^2\mu \tilde{n}}{3}\right)\leq \mydfrac{\delta}{4}	
	\end{split}
	\] 
	where the last but one inequality is the upper tail Chernoff bound, while the last inequality is granted by the assumption $\tilde{n}\geq \frac{3\mu}{(\rho-\mu)^2}\log{\frac{4}{\delta}}$.   
	The proof follows by \eqref{eq:lastfund} and \eqref{eq:llfund} after observing that  $\Pm_{n,m}\left(G\in \prd^*_n\right)=1-\Pm_{n,m}\left(\ap{L}{s}(G)\leq \rho\right)$.
\end{claimproof}
\begin{claimproof}{\ref{com:iia}}
Observe in the first place, that when $p=1/2$, the measure $\Pm_{n,\frac{1}{2}}$ assigns the same probability to the graphs in $\GG_n$. Therefore, $\Pm_{n,\frac{1}{2}}$ is the uniform measure on $\GG_n$. Borrowing the reasoning from the proof of Theorem \ref{thm:appl}, we see that the first of the two statements is equivalent to the statement 
$$\nn\geq\mydfrac{6s}{(s-2\rho)^2}\log{\frac{1}{\delta}}\Longrightarrow \Pm_{n,\frac{1}{2}}(\ma{\xi}\geq \rho)\leq \delta.$$
Observe that $\Pm_{n,\frac{1}{2}}(\ma{\xi}\geq \rho)\leq \Pm_{n,\frac{1}{2}}(\ma{\vartheta}\geq \rho)$ because $\ma{\xi} \leq \ma{\vartheta}$ holds almost surely by Lemma \ref{lem:degree_as}. Hence, after recalling that $\nn\ma{\vartheta}\sim\bin{\nn s}{1/2}$, the thesis follows by applying the upper tail Chernoff bound to $\nn\ma{\vartheta}$ with $\epsilon=\frac{2\rho-s}{s}$. Similarly, the second statement is equivalent to
$$\nn\leq \delta e^{(s-2\rho)^2/4s}\Longrightarrow \Pm_{n,\frac{1}{2}}(\ma{\xi}\leq \rho)\leq \delta.$$ 
Since by Lemma~\ref{lem:easy}, it almost surely holds that $\ma{\zeta} \leq \ma{\xi}$, by applying the Union Bound as in Lemma \ref{lem:degree_as}, yields $\Pm_{n,\frac{1}{2}}(\ma{\xi}\leq \rho)\leq \nn\Pm_{n,\frac{1}{2}}(\ma{\zeta}\leq \rho)$. The thesis now follows by applying the lower tail Chernoff bound to $\ma{\zeta}$ with $\epsilon=\frac{s-2\rho}{s}$ since $\ma{\zeta}\sim\bin{s}{1/2}$. 
\end{claimproof}
\noindent The Theorem is thus completely proved.
\end{theoremproof}

\section{Conclusion}
In this work we investigated the critical patch size problem in a
discrete setting, modeling habitats as random graphs with a subset of
vertices acting as sinks. We showed that population persistence is
governed by the Dirichlet eigenvalue of a Laplacian-type operator,
providing a natural discrete counterpart of classical
reaction-diffusion models with absorbing boundaries.

For sequences of Erd\"os-R\'enyi random habitats with extensive sink sets,
we established probabilistic bounds and proved an almost sure
asymptotic characterization of the principal Dirichlet eigenvalue in
the large-network regime. These results yield sharp spectral
thresholds determining population survival with high probability and
allow the identification of a minimal habitat size ensuring
persistence in random environments.

Several directions remain open, including the extension to structured
or correlated networks, different growth term, as in the Allee effect,
and non-asymptotic regimes relevant for finite-size habitats. These
questions may contribute to a broader understanding of persistence
phenomena in complex spatially structured systems.

\section*{Funding}
The author(s) declare that financial support was received for the
research, authorship, and/or publication of this article. The funding
for this study was partially provided by the grant PRIN 20223R9W7H
``Pre-clinical and clinical study on pediatric Traumatic Brain Injury
and intranasal Nerve Growth Factor: analysis on cortico-striatal
connectivity by network propagation modeling, neuronal tracing, and
chemogenetics'' and IAC-CNR project DIT.AD021.161 ``Analisi
probabilistica di dataset biologici e network dynamics''.
\bibliographystyle{plain}
\bibliography{biblio}
\end{document}